\documentclass{amsart}

\usepackage{amsmath,amssymb}
\usepackage{tikz-cd}
\usepackage{enumitem}

\usepackage[pdftex,
            pdfauthor={Joshua Mundinger},
            pdftitle={Hochschild cohomology of differential operators in positive characteristic}
			]{hyperref}

\usepackage[
	backend=bibtex,
	style=alphabetic,
	citestyle=alphabetic,
	maxnames=100, %list up to 100 authors
    doi=false,
]{biblatex}
\DeclareNameAlias{default}{family-given}

\usepackage{amsthm}

\theoremstyle{plain}
\newtheorem{theorem}{Theorem}[section]
\newtheorem{lemma}[theorem]{Lemma}
\newtheorem{proposition}[theorem]{Proposition}
\newtheorem{corollary}[theorem]{Corollary}
\newtheorem*{theorem*}{Theorem}

\theoremstyle{definition}
\newtheorem{example}[theorem]{Example}

\newtheorem{definition}[theorem]{Definition}
\theoremstyle{remark}
\newtheorem{remark}[theorem]{Remark}

\renewcommand{\AA}{\mathcal A}
\newcommand{\MM}{\mathcal M}

\newcommand{\DD}{\mathcal D}
\newcommand{\OO}{\mathcal O}

\newcommand{\modcat}{\mathrm{-mod}}

\newcommand{\series}[1]{[\! [ #1]\!]}

\DeclareMathOperator{\Spec}{Spec}
\DeclareMathOperator{\Hom}{Hom}
\DeclareMathOperator{\End}{End}
\DeclareMathOperator{\Ext}{Ext}

\DeclareMathOperator{\sEnd}{\underline{End}}
\DeclareMathOperator{\Tot}{Tot}

\title[
    Hochschild cohomology of differential operators
]{
    Hochschild cohomology of differential operators in positive characteristic
}
\author{Joshua Mundinger}
\address{
    University of Wisconsin-Madison\\ 
    Madison, WI, USA
}
\email{jmundinger@wisc.edu}
\date{April 16, 2024}


\begin{document}

\begin{abstract}
    For $k$ a field of positive characteristic and $X$ a smooth variety over $k$, we compute the Hochschild cohomology of Grothendieck's differential operators on $X$. The answer involves the derived inverse limit of the Frobenius acting on the cohomology of the structure sheaf of $X$.
\end{abstract}

\maketitle
\vspace{-1em}
\section{Introduction}

Algebras of polynomial differential operators are objects used throughout mathematics, sitting at the interface of analysis, geometry, topology, and algebra. 
Over a field of characteristic zero,
the algebra of differential operators $\DD_X$ on a variety $X$ consists of operators locally of the form 
$\sum_I f_I \frac{\partial^{i_1}}{\partial^{i_1} x_1}\cdots \frac{\partial^{i_n}}{\partial^{i_n} x_n}$
where $x_1,\ldots, x_n$ are local coordinates and $f_I$ are regular functions on $X$.
%The study of monodromy of systems of polynomial differential equations in the 19th century led to the Riemann-Hilbert correspondence, relating the topology of $X$ to modules over $\DD_X$.
%Beilinson and Bernstein's localization theorem connects representations of semisimple Lie algebras to $\DD$-modules on the associated flag variety \cite{beilinsonbernstein81}, relating representation theory to the geometry and topology of the flag variety.
One can also consider differential operators over a field of positive characteristic; then the definition must be modified to account for the identity $(d/dx)^px^n = 0$ in characteristic $p$. In this paper, we study differential operators with divided powers, formulated by Grothendieck in connection with crystalline cohomology. This algebra includes operators such as $(p!)^{-1}(d/dx)^p$. Precise definitions are given in §\ref{subsection: differential-operators} below. 

Hochschild cohomology is a fundamental invariant of an associative algebra $A$. As a ring, the Hochschild cohomology $HH^*(A)$ coincides with self-extensions of the $A$-bimodule $A$. Gerstenhaber showed that the bar complex of the $A$-bimodule $A$ controls deformations of the product on $A$ \cite{gerstenhaber64}. Since then, Hochschild cohomology has attracted considerable interest due to its connection to deformations.

In this note, we prove:

\begin{theorem*}[see Theorem \ref{theorem: main}]
    Let $k$ be a field of positive characteristic, $X$ a smooth variety over $k$, 
    and $\DD_X$ the ring of differential operators on $X$.
    Then there are short exact sequences 
    \[ 
        0 \to R^1\varprojlim_r H^{m-1}(X,\OO_X)^{(r)} \to HH^m(\DD_X) \to \varprojlim_r H^m(X,\OO_X)^{(r)} \to 0
    \]
    for all $m \geq 0$, where the inverse limit is taken over the sequence of relative Frobenius maps $Fr: H^\ast(X,\OO_X)^{(r+1)} \to H^\ast(X,\OO_X)^{(r)}$.
\end{theorem*}

We also consider the cup product structure on $HH^*(\DD_X)$ in §\ref{section: cup-product}. There it is shown that the map $HH^*(\DD_X) \to \varprojlim_r H^m(X,\OO_X^{(r)})$ is a map of graded rings.

Over the field $\mathbb C$ of complex numbers, it follows from work of Wodzicki that $HH^*(\DD_{X}) \cong H^\ast_{dR}(X/\mathbb C)$ for smooth affine $X$ \cite{wodzicki87}. In fact, in \textit{loc cit.} Wodzicki calculates the Hochschild homology as $HH_*(\DD_{X/\mathbb C}) \cong H_{dR}^{2\dim X - \ast}(X/\mathbb C)$, reflecting that $\DD_{X/\mathbb C}$ is a Calabi-Yau algebra of dimension $2 \dim X$. One year later, Wodzicki calculated the Hochschild homology of $\DD_X$ for $X$ smooth affine over $k$ of positive characteristic \cite{wodzicki88}. He showed that $HH_*(\DD_X)$ is concentrated in degree $n = \dim X$ and $HH_n(\DD_X) = \Omega^n_X / B_\infty \Omega^n_X$, where $B_\infty \Omega^n_X$ are the \emph{potentially exact forms}. Unlike in characteristic zero, Wodzicki's method does not apply to Hochschild cohomology. Our results show $\DD_X$ over a field of positive characteristic is not Calabi-Yau, as for affine $X$, $HH^*(\DD_X)$ will have amplitude exactly $[0,1]$.

If $X$ is not affine, then $\DD_X$ is not an associative algebra but rather a sheaf of such; this makes defining an appropriate category of bimodules more difficult. Hochschild cohomology of ringed spaces and, more generally, abelian categories has been developed by Lowen and Van den Bergh \cite{lowenvandenbergh05}, following earlier work by Gerstenhaber and Schack on cohomology of diagrams of associative algebras \cite{gerstenhaberschack83, gerstenhaberschack88}. Lowen and Van den Bergh showed that if $\AA$ is a sheaf of algebras on $X$ which is locally acyclic with respect to a basis $\mathcal B$, then the Gerstenhaber-Schack cohomology of the diagram $\AA_{\mathcal B}$ coincides with Hochschild cohomology of the category of $\AA$-modules \cite[§7.3]{lowenvandenbergh05}. This paper uses Gerstenhaber and Schack's definition of Hochschild cohomology of a diagram (see §\ref{section: gerstenhaber-schack-complex}); as $\DD_X$ is quasi-coherent, our computation thus agrees with categorical Hochschild cohomology.

We should point out that Ogus obtains very similar formulas for $\Ext^*_{\DD_X}(\OO_X,\OO_X)$ in \cite{ogus75}. The relationship is explained in Remark \ref{remark: ogus}. 

\vspace{.2in}
\noindent\textbf{Acknowledgments.}
{
    The author was supported by NSF Graduate Research Fellowship DGE 1746045. 
    The author thanks Luc Illusie for pointing out the paper \cite{ogus75} and Tim Porter for the reference \cite{abelsholz93}.
}

\section{Preliminaries}

\subsection{Differential operators}
\label{subsection: differential-operators}

Let $k$ be a field.
A \emph{smooth variety} $X$ over $k$ will mean a morphism of schemes $X \to \Spec k$ which is smooth, separated, and of finite type. We recall Grothendieck's definition of differential operators on a smooth variety, following \cite[§2.1]{berthelot-ogus-crystalline}.

\begin{definition}
    \cite[Definition 2.1]{berthelot-ogus-crystalline}
    For a smooth variety $X$ over $k$, let $I$ be the ideal of the diagonal in $X \times_k X$, that is, $I = \langle f \otimes 1 - 1 \otimes f \rangle \subseteq \OO_X \otimes_k \OO_X$.
    A $k$-linear map of sheaves $h : \OO_X \to \OO_X$ is a \emph{differential operator of order $m$} if and only if the linearization $h: \OO_X \otimes_k \OO_X \to \OO_X$ is annihilated by $I^{m+1}$.
\end{definition}
A map $h$ is a differential operator of order $m$ if and only if for all local sections $f_0,f_1,\ldots, f_m$ of $\OO_X$, we have $[f_0,[f_1,[\ldots,[f_m,h]]]] = 0$.
A degree zero differential operator is $\OO_X$-linear and thus is a section of $\OO_X$.
Differential operators form a filtered sheaf of rings $\DD_X$ with a map $\OO_X \to \DD_X$ from the inclusion of degree zero differential operators.
As a left or right $\OO_X$-module, the sheaf $\DD_X$ is quasi-coherent.

The sheaf $\DD_X$ contains differential operators ``with divided powers.'' For example, if $\partial = d/dt$ is the usual derivative acting on $k[t]$, the operator $\partial^q$ is divisible by $q!$, as $\partial^q t^m = q! \binom{m}{q}t^{m-q}$. Thus $\frac{1}{q!} \partial^q: t^m \mapsto \binom{m}{q}t^{m-q}$ makes sense, regardless of the characteristic of $k$.
The following proposition states that $\DD_X$ locally has a basis of such operators:

\begin{proposition}\cite[Proposition 2.6]{berthelot-ogus-crystalline} \label{proposition: basis-of-divided-powers}
    Suppose $X$ is a smooth variety and $\{x_1,\ldots, x_n\} \subseteq \Gamma(X,\OO_X)$ such that $\{dx_1,\ldots, dx_n\}$ is a basis for $\Omega^1_X$. Then as either a left or a right $\OO_X$-module,
    $\DD_X$ has a basis of the form 
    \[ \{ \partial_1^{(q_1)}\partial_2^{(q_2)}\cdots \partial_n^{(q_n)} \mid q_1,q_2,\ldots, q_n \geq 0\} \]
    where $\partial_i^{(q_i)}$ are differential operators such that
    \begin{itemize}
        \item $\partial_i^{(0)} = 1$ for all $i$;
        \item the operators $\{\partial_i^{(1)} = \partial_i\}_{i=1}^n$ are derivations dual to the 1-forms $\{dx_i\}_{i=1}^n$;
        \item $\partial_i^{(q)}\partial_i^{(q')} = \binom{q+q'}{q} \partial_i^{(q+q')}$ for all $i$ and $q,q' \geq 0$;
        \item the following commutator relations hold:
        \begin{align*}
            [\partial_i^{(q_i)}, \partial_j^{(q_j)}] &= 0,&[\partial_i^{(q_i)},x_j] &= \begin{cases}
                0 & i \neq j\\
                \partial_i^{(q_i - 1)} & i = j .
            \end{cases}  
        \end{align*}
    \end{itemize}
\end{proposition}
We should think of $\partial_i^{(q)}$ as $\frac{1}{q!}(d/dx_i)^q$, and indeed the commutator relations imply $\partial_i^{(q)}(x_i^m) = \binom{m}{q}x_i^{m-q}$.

\subsection{Frobenius twists and matrix subalgebras}
\label{subsection: frob-twist}

In this section, assume that $k$ has characteristic $p > 0$. 

Consider the setting of Proposition \ref{proposition: basis-of-divided-powers}. The derivations $\{\partial_i\}_{i=1}^n$ do not generate the algebra $\DD_X$ over $\OO_X$. Instead, they generate an associative algebra of finite type over $\OO_X$, centralizing the $p$th power of any element of $\OO_X$.  
To deal with rings of $p$th powers, we use the standard language of Frobenius twists.
For any scheme $S$ of characteristic $p$, let $F_S: S \to S$ be the absolute Frobenius morphism on $S$, which is the identity on $|S|$ and acts on functions by $f \mapsto f^p$.

\begin{definition}
    The \emph{Frobenius twist} $X^{(1)}$ of $X\to \Spec k$ is the base change $X \times_{\Spec k} (\Spec k, F_{\Spec k})$ of $X$ along the absolute Frobenius of $\Spec k$.
\end{definition}
The Frobenius twist fits into a diagram 
\begin{equation*}
\begin{tikzcd}[cramped]
    X \arrow[rd] \arrow[r, "Fr_{X/k}"] & X^{(1)} \arrow[r] \arrow[d] & X \arrow[d] \\
        & \Spec k \arrow[r, "F_{\Spec k}"]           & \Spec k          
\end{tikzcd}
\end{equation*}
%\begin{diagram}
%    X   & \rTo^{Fr_{X/k}}  & X^{(1)}   & \rTo      & X \\
%        & \rdTo & \dTo      &           & \dTo \\
%        &       & \Spec k   & \rTo^{F_k}& \Spec k 
%\end{diagram}
where the composition $X \to X^{(1)} \to X$ is the absolute Frobenius $F_X$ of $X$.
The induced map $Fr_{X/k}: X \to X^{(1)}$ over $\Spec k$ is called the \emph{relative Frobenius}.
When there is no risk of confusion, we will write $Fr$ instead of $Fr_{X/k}$.
If $X$ is a smooth variety over $k$, then $Fr$ is a homeomorphism and the induced map $Fr: \OO_{X^{(1)}} \to Fr_\ast \OO_X$ is an inclusion making $Fr_\ast \OO_X$ a finite free $\OO_{X^{(1)}}$-module of rank $p^{\dim X}$ \cite[1.1.1 Lemma]{brionkumar}.
We may identify the image of $\OO_{X^{(1)}}$ in $\OO_X$ with the Frobenius twist $\OO_X^{(1)}$ of the sheaf $\OO_X$.
We will also write $Fr^r: X \to X^{(r)}$ to mean the $r$-fold composition of relative Frobenius maps $X \to X^{(1)} \to \cdots \to X^{(r)}$.

\begin{definition}
    The subalgebra $\DD^r_X$ is the centralizer of $\OO_{X}^{(r)}$ in $\DD_X$.
\end{definition}

\begin{proposition} \cite[Lemma 3.3]{chase74}
    \label{proposition: subalgebra}
    If $X$ is a smooth variety over $k$,
    then the action of $\DD_X$ on $\OO_X$ induces an isomorphism $\DD^r_X \to \sEnd_{X^{(r)}}(Fr^r_\ast\OO_X)$,
    and $\DD_X = \bigcup_{r\geq 0} \DD^r_X$.
\end{proposition}
\begin{proof}
    The map $\DD_X^{r} \to \sEnd_{X^{(r)}}(Fr^r_\ast\OO_X)$ is injective.
    For $r \geq 1$, 
    Let $I^{(r)} = \langle f^{p^r}\otimes 1 - 1 \otimes f^{p^r} \rangle \subseteq \OO_X \otimes_k \OO_X$.
    Since $X$ is of finite type, the chains of ideals $I \supseteq I^2 \supseteq \cdots$ and $I \supseteq I^{(1)} \supseteq I^{(2)} \supseteq \cdots$ are cofinal.
    Hence every $\OO_{X^{(r)}}$-linear map of $Fr^r_\ast \OO_X$ is a differential operator,
    so $\DD^r_X\to \sEnd_{X^{(r)}}(Fr^r_\ast\OO_X)$ is surjective.
    Cofinality also implies every differential operator is linear over some $\OO_X^{(r)}$, so $\DD_X = \bigcup_{r\geq 0}\DD^r_X$.
\end{proof}

%\begin{remark}
%    If $X$ has local coordinates $x_1,\ldots, x_n$, then $\DD^r_X$ is spanned by those $\partial_1^{(q_1)}\cdots \partial_n^{(q_n)}$ such that $q_i < p^r$ for all $i$.
%    However, we will not need this fact.
%\end{remark}

\begin{remark}
    The subalgebras $\DD^r_X$ are Morita equivalent to $\OO_X^{(r)}$, realizing Katz's equivalence of $\DD_X$-modules with stratified sheaves on $X$ \cite{gieseker75}.
\end{remark}

\section{The Gerstenhaber-Schack complex}
\label{section: gerstenhaber-schack-complex}

Let $k$ be any field. We now describe the Gerstenhaber-Schack complex for computing Hochschild cohomology \cite[(6.5)]{gerstenhaberschack88}.
The construction begins with a poset $J$, a presheaf of $k$-algebras $\AA$ on $J$, and a presheaf of $\AA$-bimodules $\MM$.
Let $N(J)$ be the nerve of the poset $J$ in the sense of category theory: it is the simplicial complex with $i$-simplices the chains $\sigma = (U_0 < \cdots < U_i)$ for $U_0,\ldots,U_i \in J$. Any such simplex $\sigma$ has a maximal element $\max \sigma = U_i$ and a minimal element $\min\sigma = U_0$. Then the Gerstenhaber-Schack complex is the double complex 
with cochains 
\begin{equation}
    \label{eq: gerstenhaber-schack-complex}
    C^{i,j}(\AA, \MM) = \prod_{\sigma \in N(J)^i} C^j(\AA(\max \sigma), \MM(\min \sigma)),
\end{equation}
where $C^\ast(-,-)=\Hom_k(-^{\otimes \ast},-)$ are the usual Hochschild cochains.
The horizontal differential is the differential on simplicial cochains on $N(J)$ with respect to the system of local coefficients $\sigma \mapsto C^j(\AA(\max \sigma),\MM(\min \sigma))$.
The vertical differential is the Hochschild differential. Then Hochschild cohomology of $(\AA,\MM)$ is defined by
\[ HH^*(\AA,\MM) = H^*(\Tot C^{\ast,\ast}(\AA,\MM)).\]
In case $J$ is a single point, the Gerstenhaber-Schack complex coincides with the usual Hochschild complex of an associative algebra.

We now consider the standard filtration $F^\ell C^{i,j} = 0$ if $i< \ell$ and $F^\ell C^{i,j} = C^{i,j}$ if $i \geq \ell$.
The associated spectral sequence yields:
\begin{proposition}\label{prop: spectral sequence}
    Let $J$ be a poset, $\AA$ a presheaf of associative $k$-algebras on $J$, and $\MM$ a presheaf of $\AA$-bimodules.
    Consider the local system on $N(J)$ with coefficients 
    \[
        \underline{HH}^j(\AA,\MM)(\sigma) = HH^j(\AA(\max \sigma),\MM(\min \sigma)).
    \]
    There is a first-quadrant spectral sequence with $E_2$-page
    \[ E_2^{i,j} = H^i(N(J), \underline{HH}^j(\AA,\MM))\]
    converging to $HH^{i+j}(\AA,\MM)$.
\end{proposition}

In our situation, the algebra $\DD_X$ is the union of subalgebras $\DD^r_X$. The following Lemma relates Hochschild cohomology of $\DD^r_X$ and $\DD_X$.

\begin{lemma}\label{lemma: inverse-limit}
    Let $J$ be a poset,
    $\AA$ a presheaf of associative $k$-algebras on $J$ with a complete filtration by subalgebras $\AA^0\subseteq \AA^1\subseteq \cdots \subseteq \AA$. Then for any presheaf of $\AA$-bimodules $\MM$, there are short exact sequences 
    \[ 0 \to R^1 \varprojlim_r HH^{m-1}(\AA^r,\MM) \to HH^m(\AA,\MM) \to \varprojlim_r HH^m(\AA^r,\MM) \to 0\]
    for all $m \geq 0$.
\end{lemma}
\begin{proof}
    Fix an $i$-simplex $\sigma \in N(J)$, and let $U_0 = \min \sigma, U_i = \max \sigma$.
    Then we have $\AA(U_i) = \cup_r \AA^r(U_i)$.
    Since $k$ is a field, the maps 
    \begin{equation*}
       \Hom_k(\AA(U_i)^{\otimes j}, \MM(U_0)) 
       \to \Hom_k(\AA^{r+1}(U_i)^{\otimes j},\MM(U_0))
       \to \Hom_k(\AA^{r}(U_i)^{\otimes j},\MM(U_0))
    \end{equation*}
    are surjective for all $j \geq 0$.
    Hence, the sequence of complexes $C^{i,j}(\AA^r,\MM)$ satisfies the Mittag-Leffler condition on cochains and 
    \[ C^{i,j}(\AA,\MM) = \varprojlim_r C^{i,j}(\AA^r,\MM).\] Now the Proposition follows from \cite[Theorem 3.5.8]{weibel}.
\end{proof}

\section{Hochschild cohomology of differential operators}

For the rest of the paper, the field $k$ is of characteristic $p > 0$ and $X$ is a smooth variety over $k$. As $X$ is separated, the intersection of two affine opens is an affine open. 
Hence there exists a covering of $X$ by nonempty open affines closed under finite intersections; let $\mathfrak U$ be such a covering.
We will compute the cohomology of the Gerstenhaber-Schack complex for $\DD_X$ with respect to the poset $\mathfrak U$. 
Here is a sketch of the argument: we will first compute $HH^*(\DD^r_X,\DD_X)$ locally, then apply the spectral sequence of Proposition \ref{prop: spectral sequence} to compute $HH^*(\DD^r_X,\DD_X)$.
Finally, Lemma \ref{lemma: inverse-limit} will relate $HH^*(\DD^r_X,\DD_X)$ with $HH^*(\DD_X)$.

Given affine $U$, let $\DD(U) = \Gamma(U,\DD_U)$, and similarly for $\DD^r, \OO,$ etc..

\begin{proposition}\label{proposition: pd-dr}
    Let $V \subseteq U$ be smooth affine varieties over $k$. Then 
    the inclusion $\OO(V^{(r)}) \to \DD(V)$ induces 
    \[
        HH^0(\DD^r(U), \DD(V)) \cong \OO(V^{(r)})
    \]
    and $HH^m(\DD^r(U), \DD(V)) = 0$ for all $m > 0$.
\end{proposition}
\begin{proof}
    By Proposition \ref{proposition: subalgebra}, 
    the algebra $\DD^r(U)$ is the matrix ring $\End_{\OO(U^{(r)})}(\OO(U))$.
    It is well-known that Hochschild cohomology is Morita-invariant.
    Under the Morita equivalence $\DD^r(U)\modcat \simeq \OO(U^{(r)})\modcat$, the $\DD^r(U)$-bimodule $\DD(V)$ is sent exactly to $\DD(V^{(r)})$.
    This is seen as follows: for affine $U$, the inclusion $\OO(U^{(r)}) \to \OO(U)$ is split as $\OO(U^{(r)})$-modules \cite[1.1.6 Proposition]{brionkumar}.
    If $e \in \DD^r(U)$ is the idempotent corresponding to the splitting, then it is easily checked that the map $e\DD(V)e \to \End_k(\OO(V^{(r)}))$ induces an isomorphism $e \DD(V)e \cong \DD(V^{(r)})$.

    Thus, it suffices to show the claim when $r=0$,
    that is, to compute Hochschild cohomology of $(\OO(U),\DD(V))$. The claim is local, so we may assume that $U$ has local coordinates $x_1,\ldots, x_n$ such that $\{x_i \otimes 1 - 1 \otimes x_i\}$ generate the ideal of the diagonal in $U \times U$. 
    Let $P_\ast \to \OO(U)$ be the Koszul complex for $\OO(U)$ as an $\OO(U)$-bimodule corresponding to the regular sequence $x_1,\ldots, x_n$. Then $HH^\ast(\OO(U),\DD(V))$ is computed by the complex $\Hom_{\OO(U)\otimes \OO(U)}(P_\ast, \DD(V))$.   
    Using the basis of Proposition \ref{proposition: basis-of-divided-powers}, the complex above is the tensor product over $\OO(V)$ of the two-term complexes 
    \begin{equation}\label{equation: pd-de-Rham-differential}
    %\begin{diagram}
    %    \OO(V) \langle \partial_i\rangle  & \rTo^{[x_i,-]} & \OO(V)\langle \partial_i\rangle
    %\end{diagram}
    \begin{tikzcd}
        \OO(V) \langle \partial_i\rangle 
        \arrow[r,"{[x_i,-]}"] 
        & \OO(V)\langle \partial_i\rangle
    \end{tikzcd}
    \end{equation} 
    for $1 \leq i \leq n$,
    where $\OO(V)\langle - \rangle$ means free divided power algebra. The differential \eqref{equation: pd-de-Rham-differential} is surjective with kernel $\OO(V)$. Now the Künneth formula finishes the proof.
\end{proof}

\begin{remark}
    The complex constructed in the proof above coincides with the PD de Rham complex for a free divided power algebra. See \cite[Lemma 6.11]{berthelot-ogus-crystalline}.
\end{remark}

\begin{proposition}\label{proposition: cohomology-of-subalgebra}
    For a smooth variety $X$ over $k$,
    there is an isomorphism
    \[
        HH^m(\DD^r_X,\DD_X) \cong H^m(X,\OO_X^{(r)})
    \]
    for all $m \geq 0$
    such that the following diagram commutes:
    %\begin{diagram}
    %    HH^m(\DD^{r+1}_X,\DD_X) & \rTo^\simeq & H^m(X,\OO_X^{(r+1)}) \\
    %    \dTo & & \dTo^{Fr} \\
    %    HH^m(\DD^r_X,\DD_X) & \rTo^\simeq & H^m(X,\OO_X^{(r)}).
    %\end{diagram}
    \begin{equation*}
    \begin{tikzcd}
        HH^m(\DD^{r+1}_X,\DD_X) \arrow[r, "\sim"] \arrow[d]
        & H^m(X,\OO_X^{(r+1)}) \arrow[d, "Fr"] \\
        HH^m(\DD^{r}_X,\DD_X) \arrow[r, "\sim"] 
        & H^m(X,\OO_X^{(r)})
    \end{tikzcd}
    \end{equation*}
\end{proposition}
\begin{proof}
    Consider the spectral sequence of Proposition \ref{prop: spectral sequence} for $(\DD^r_X,\DD_X)$. By Proposition \ref{proposition: pd-dr},
    for any pair $V\subseteq U$ in $\mathfrak U$,
    we have $HH^0(\DD^r(U),\DD(V)) = \OO(V)^{(r)}$ and $HH^{j}(\DD^r(U),\DD(V)) = 0$ for $j > 0$. 
    Thus, the spectral sequence has $E_2$-page 
    \[ E_2^{ij} = \begin{cases}
            H^i(N(\mathfrak U),\OO_X^{(r)}) & j = 0 \\
            0 & j > 0,
        \end{cases}
    \]
    and so collapses at $E_2$.
    
    We now relate $N(\mathfrak U)$, the nerve in the sense of category theory, to \v{C}ech cohomology with respect to $\mathfrak U$.
    Let $N_{\check{C}ech}(\mathfrak U)$ be the \v{C}ech nerve of $\mathfrak U$.
    As $\mathfrak U$ is closed under intersection, Theorem 1.4 of \cite{abelsholz93} implies that the map $N_{\check{C}ech}(\mathfrak U) \to \mathfrak U$ sending $(U_0,\ldots, U_i) \in \mathfrak U^{i+1}$ with nonempty intersection to $U_0 \cap \cdots \cap U_i$ induces a homotopy equivalence $N_{\check{C}ech}(\mathfrak U) \to N(\mathfrak U)$.
    This map induces 
    \[ H^i(N(\mathfrak U),\OO_X^{(r)}) \cong H^i(N_{\check{C}ech}(\mathfrak U),\OO_X^{(r)}) = \check{H}^i(\mathfrak U,\OO_X^{(r)}).\]
    The sheaf $\OO_X^{(r)}$ is quasi-coherent on $X^{(r)}$
    and the $r$-fold relative Frobenius $Fr^r: X \to X^{(r)}$ is a homeomorphism, so $\OO_X^{(r)}$ is acyclic for every affine open subset of $X$.
    As $\mathfrak U$ consists of affine open subsets of $X$ and is closed under intersection, we obtain 
    $E_2^{i,0} = H^i(X,\OO_X^{(r)})$. This proves the claim.
    %Compatibility with Frobenius follows from functoriality of the spectral sequence.
\end{proof}

\begin{theorem}\label{theorem: main}
    Let $k$ be a field of positive characteristic and $X$ be a smooth variety over $k$.
    Then there are short exact sequences 
    \[ 0 \to R^1\varprojlim_r H^{m-1}(X,\OO_X^{(r)}) \to HH^m(\DD_X) \to \varprojlim_r H^m(X,\OO_X^{(r)}) \to 0\]
    for all $m \geq 0$,
    where the limit is taken over the sequence of relative Frobenius maps $Fr: X^{(r)} \to X^{(r+1)}$.
\end{theorem}
\begin{proof}
    Combine Proposition \ref{proposition: cohomology-of-subalgebra} with Lemma \ref{lemma: inverse-limit}.
\end{proof}

\begin{remark}\label{remark: ogus}
    In \cite[(2.4),(3.6)]{ogus75}, Ogus shows that there is an exact sequence 
    \[ 0 \to R^1 \varprojlim_r H^{m-1}(X,\OO_X^{(r)}) \to \Ext^m_{\DD_X}(\OO_X,\OO_X) \to \varprojlim_r H^m(X,\OO_X^{(r)}) \to 0\]
    for all $m \geq 0$. The connection with Theorem \ref{theorem: main} is this: for any abelian category $\mathcal C$ and object $C$ of $\mathcal C$ there is the \emph{Chern character} $ch_C: HH^*(\mathcal C) \to \Ext^*_{\mathcal C}(C,C)$ \cite{lowenvandenbergh05}.
    In fact, the Chern character induces an isomorphism 
    \begin{equation} \label{eq: chern-character}
        ch_{\OO_X}: HH^*(\DD_X) \cong \Ext^*_{\DD_X}(\OO_X, \OO_X).
    \end{equation}
    When $\mathcal C = A\modcat$ for an associative algebra $A$ and $C \in A\modcat$, the characteristic morphism is implemented by the morphism of Hochschild complexes $C^*(A,A) \to C^*(A,\End_k(C))$. 
    It may be checked for smooth affine $X$ that \eqref{eq: chern-character} is an isomorphism, e.g.\ by suitably modifying the argument of Proposition \ref{proposition: cohomology-of-subalgebra}.
    The statement for a general smooth variety $X$ may be reduced to the affine case by checking that the Chern character in \cite{lowenvandenbergh05}, is suitably natural, e.g.\ with respect to the Mayer-Vietoris sequence for Hochschild cohomology.
\end{remark}

\section{Examples}

\begin{example}
    Let $X = \mathbb A^1_k$, so that $\OO(X) = k[t]$.
    Then $\OO(X^{(r)}) = k[t^{p^r}]$
    and the relative Frobenius $X^{(r)} \to X^{(r+1)}$ is induced by the inclusion $k[t^{p^r}] \to k[t^{p^{r+1}}]$. Theorem \ref{theorem: main} implies that $HH^0(\DD_{\mathbb A^1}) = k$, 
    \begin{equation*}
        HH^1(\DD_{\mathbb A^1}) \cong R^1 \varprojlim_r k[t^{p^r}],
    \end{equation*}
    and $HH^m(\DD_{\mathbb A^1}) = 0$ when $m > 1$.
    Let us compute $R^1\varprojlim_r k[t^{p^r}]$.
    The short exact sequence of diagrams 
    %\[ 0 \to \{ k[t^{p^r}]\}_r \to \{k[t]\}_r \to \{k[t]/k[t^{p^r}]\}_r \to 0\]
    % https://q.uiver.app/#q=WzAsNSxbMCwwLCIwIl0sWzEsMCwiXFx7a1t0XntwXnJ9XVxcfV9yIl0sWzIsMCwiXFx7a1t0XVxcfV9yIl0sWzMsMCwiXFx7a1t0XS9rW3Ree3Becn1dXFx9X3IiXSxbNCwwLCIwIl0sWzAsMV0sWzEsMl0sWzIsM10sWzMsNF1d
\[\begin{tikzcd}
	0 & {\{k[t^{p^r}]\}_r} & {\{k[t]\}_r} & {\{k[t]/k[t^{p^r}]\}_r} & 0
	\arrow[from=1-1, to=1-2]
	\arrow[from=1-2, to=1-3]
	\arrow[from=1-3, to=1-4]
	\arrow[from=1-4, to=1-5]
\end{tikzcd}\]
    gives a medium-length exact sequence of derived inverse limits
    % https://q.uiver.app/#q=WzAsNixbMCwwLCIwIl0sWzEsMCwiayJdLFsyLDAsImtbdF0iXSxbMywwLCJcXHZhcnByb2psaW1cXGxpbWl0c19yIGtbdF0va1t0XntwXnJ9XSJdLFs0LDAsIlJeMVxcdmFycHJvamxpbVxcbGltaXRzX3Iga1t0XntwXnJ9XSJdLFs1LDAsIjAiXSxbMCwxXSxbMSwyXSxbMiwzXSxbMyw0XSxbNCw1XV0=
\[\begin{tikzcd}
	0 & k & {k[t]} & {\varprojlim\limits_r k[t]/k[t^{p^r}]} & {R^1\varprojlim\limits_r k[t^{p^r}]} & 0
	\arrow[from=1-1, to=1-2]
	\arrow[from=1-2, to=1-3]
	\arrow[from=1-3, to=1-4]
	\arrow[from=1-4, to=1-5]
	\arrow[from=1-5, to=1-6]
\end{tikzcd}.\]
    %\begin{align*}
    %    0 \to k \to k[t] \to \varprojlim_r k[t]/k[t^{p^r}] \to R^1\varprojlim_r k[t^{p^r}] \to 0 \to 0 \to 0.
    %\end{align*}
    Hence $R^1 \varprojlim_r k[t^{p^r}]$ is identified with $(\varprojlim_r k[t]/k[t^{p^r}])/k[t]$. The space $\varprojlim_r k[t]/k[t^{p^r}]$ embeds into $k\series{t}$ as those power series which may be written as a formal sum $f= \sum_{r=0}^\infty f_r$ for $f_r \in k[t^{p^r}]$.
    The identification $HH^1(\DD(X)) \cong R^1\varprojlim_r \OO(X^{(r)})$ sends such a power series $f$ to the commutator with $f$, which is well-defined since every element of $\DD(X)$ commutes with some $\OO(X^{(r)})$.
    This explains S.\ Paul Smith's observation from 1986 that taking the commutator with the series $t + t^p + t^{p^2} + \cdots$ gives an outer derivation of $\DD(\mathbb A_k^1)$ \cite[Proposition 2.4]{smith86}.
\end{example}

\begin{example}
    Suppose that $X$ is proper over $k$, so that $H^m(X,\OO_X)$ is finite-dimensional over $k$ for all $m\geq 0$. Pullback by Frobenius is a semilinear endomorphism of $H^m(X,\OO_X)$, and thus there is a canonical splitting
    \[ H^m(X,\OO_X) = H^m_n(X,\OO_X) \oplus H^m_s(X,\OO_X)\]
    into $Fr$-stable subspaces where the Frobenius acts nilpotently and as an isomorphism, respectively \cite[p. 143]{mumford70}. Then 
    $ \varprojlim_r H^m(X,\OO_X^{(r)}) \cong H^m_s(X,\OO_X)$
    via the projection $\varprojlim_r H^m(X,\OO_X^{(r)}) \to H^m(X,\OO_X)$, while
     $R^1 \varprojlim_r H^m(X,\OO_X^{(r)}) = 0$.
    We conclude $HH^m(\DD_X) \cong H^m_s(X,\OO_X)$ when $X$ is proper over $k$.
\end{example}  

\section{Cup product}
\label{section: cup-product}

Gerstenhaber and Schack defined a cup product on \eqref{eq: gerstenhaber-schack-complex} in terms of the cup product on the Hochschild complex.
Given simplices $\tau = (U_0 < \cdots < U_\ell)$ and  $\nu= (U_\ell < \cdots < U_i)$ in $N(J)$ with $\max \tau = \min \nu$, 
define $\tau \cup \nu = (U_0 < \cdots < U_\ell < \cdots < U_i)$.
For $\AA$ a presheaf of associative algebras on $J$, $\MM$ a presheaf of $\AA$-bimodules with compatible associative product, $\alpha \in C^m(\AA,\MM)$, and $\beta \in C^n(\AA,\MM)$, define 
\begin{equation}\label{eq: cup-product}
    (\alpha \cup \beta)^\sigma = \sum_{\sigma = \tau\cup \nu } (-1)^{|\nu|(m-|\tau|)}\alpha^\tau \varphi_\AA \cup \varphi_{\MM}\beta^\nu,
\end{equation} 
where $\varphi$ refers to structure maps of the presheaves $\AA$ and $\MM$, and $\cup$ on the right-hand side is the cup product on the usual Hochschild complex \cite[§6, p. 192]{gerstenhaberschack88}. 

\begin{proposition}
    Let $J$ be a poset, $\AA$ a presheaf of associative $k$-algebras on $J$, and $\MM$ a presheaf of $\AA$-bimodules with a compatible associative product.
    Then the spectral sequence 
    \[ E_2^{i,j} = H^i(N(J),\underline{HH}^j(\AA,\MM)) \implies HH^{i+j}(\AA,\MM)\]
    of Proposition \ref{prop: spectral sequence} is a spectral sequence of algebras with the usual product on $E_2$.
\end{proposition}
\begin{proof}
    Formula \eqref{eq: cup-product} makes the Gerstenhaber-Schack double complex a differential bigraded algebra,
    and endows $E_2$ with the usual product structure.
\end{proof}

\begin{corollary}
    Let $k$ be a field of positive characteristic and $X$ be a smooth variety over $k$.
    Then the surjection $HH^\ast(\DD_X) \to \varprojlim_r H^\ast(X,\OO_X^{(r)})$ of Theorem \ref{theorem: main} is a map of rings.
\end{corollary}

%\section{End Matter}

\printbibliography

@article{abelsholz93,
  author    = {Abels, Herbert and Holz, Stephan},
  journal   = {J. Algebra},
  number    = {2},
  pages     = {310--341},
  publisher = {Elsevier},
  title     = {Higher generation by subgroups},
  volume    = {160},
  year      = {1993}
}

@book{berthelot-ogus-crystalline,
  author    = {Berthelot, Pierre and Ogus, Arthur},
  number    = {21},
  publisher = {Princeton University Press},
  series    = {Mathematical Notes},
  title     = {Notes on crystalline cohomology},
  year      = {1978}
}

@book{brionkumar,
  author    = {Brion, Michel and Kumar, Shrawan},
  number    = {231},
  publisher = {Birkh\"auser},
  series    = {Progr. Math.},
  title     = {Frobenius splitting methods in geometry and representation theory},
  year      = {2007}
}

@article{chase74,
  author  = {Chase, Stephen U.},
  doi     = {10.1080/00927877408548623},
  journal = {Comm. Algebra},
  number  = {5},
  pages   = {351-363},
  title   = {On the homological dimension of algebras of differential operators},
  volume  = {1},
  year    = {1974}
}

@article{gerstenhaber64,
  author  = {Murray Gerstenhaber},
  journal = {Ann. of Math.},
  number  = {1},
  pages   = {59--103},
  title   = {On the {D}eformation of {R}ings and {A}lgebras},
  volume  = {79},
  year    = {1964}
}

@article{gerstenhaberschack83,
  author  = {Gerstenhaber, Murray and Schack, Samuel D},
  journal = {Trans. Amer. Math. Soc.},
  number  = {1},
  pages   = {1--50},
  title   = {On the deformation of algebra morphisms and diagrams},
  volume  = {279},
  year    = {1983}
}

@article{gerstenhaberschack88,
  author  = {Gerstenhaber, Murray and Schack, Samuel D},
  journal = {Trans. Amer. Math. Soc.},
  number  = {1},
  pages   = {135--165},
  title   = {The cohomology of presheaves of algebras. I. {P}resheaves over a partially ordered set},
  volume  = {310},
  year    = {1988}
}

@article{gieseker75,
  author  = {Gieseker, David},
  journal = {Ann. Sc. Norm. Super. Pisa Cl. Sci.},
  number  = {1},
  pages   = {1--31},
  title   = {Flat vector bundles and the fundamental group in non-zero characteristics},
  volume  = {2},
  year    = {1975}
}

@article{lowenvandenbergh05,
  author  = {Lowen, Wendy and {Van den Bergh}, Michel},
  journal = {Adv. Math.},
  number  = {1},
  pages   = {172--221},
  title   = {Hochschild cohomology of abelian categories and ringed spaces},
  volume  = {198},
  year    = {2005}
}

@book{mumford70,
  author    = {Mumford, David},
  number    = {5},
  publisher = {Oxford University Press},
  series    = {Tata Institute of Fundamental Research Studies in Mathematics},
  title     = {Abelian varieties},
  year      = {1970}
}

@article{ogus75,
  author  = {Ogus, Arthur},
  journal = {Ann. Sci. Éc. Norm. Supér.},
  number  = {3},
  pages   = {295--318},
  title   = {Cohomology of the infinitesimal site},
  volume  = {8},
  year    = {1975}
}

@incollection{smith86,
  author    = {Smith, S Paul},
  booktitle = {S\'eminaire d'alg\`ebre Paul Dubreil et Marie-Paule Malliavin, 37\`eme ann\'ee (Paris, 1985)},
  number    = {1220},
  pages     = {157--177},
  publisher = {Springer},
  series    = {Lecture Notes in Math.},
  title     = {Differential operators on the affine and projective lines in characteristic {$p> 0$}},
  year      = {1986}
}

@book{weibel,
  author    = {Weibel, Charles A},
  number    = {38},
  publisher = {Cambridge University Press},
  series    = {Cambridge Stud. Adv. Math.},
  title     = {An introduction to homological algebra},
  year      = {1995}
}

@article{wodzicki87,
  author  = {Wodzicki, Mariusz},
  journal = {Duke Math. J.},
  number  = {1},
  pages   = {641--647},
  title   = {Cyclic homology of differential operators},
  volume  = {54},
  year    = {1987}
}

@article{wodzicki88,
  author  = {Wodzicki, Mariusz},
  journal = {C. R. Math. Acad. Sci. Paris Sér. I},
  number  = {6},
  pages   = {249-254},
  title   = {Cyclic homology of differential operators in characteristic {$p>0$}},
  volume  = {307},
  year    = {1988}
}

\end{document}